\documentclass[10pt, onecolumn, twoside]{IEEEtran}
\usepackage{fullpage}
\usepackage{graphicx}
\usepackage{amssymb}
\usepackage{amsfonts}
\usepackage{amsmath}
\usepackage{epstopdf}
\usepackage{dsfont}
\usepackage{color}
\usepackage{mathtools}

\newtheorem{lemma}{Lemma}[section]
\newtheorem{theorem}[lemma]{Theorem}
\newtheorem{corollary}[lemma]{Corollary}
\newtheorem{example}[lemma]{Example}
\newtheorem{definition}[lemma]{Definition}

\newtheorem{proposition}[lemma]{Proposition}
\newtheorem{remark}[lemma]{Remark}
\newenvironment{proof}{\vspace{.1cm}\noindent{\sc Proof.}\hspace{0.10cm}\,\,}{$\hfill\Box$\vspace{.3cm}}

\newcommand{\Exp}{\mathrm{Exp\,}}
\newcommand{\Ln}{\mathrm{Ln\,}}
\newcommand{\diag}{\mathrm{diag\,}}
\newcommand{\rank}{\mathrm{rank\,}}
\newcommand{\im}{\mathrm{im\,}}
\newcommand{\spa}{\mathrm{span\,}}
\newcommand{\ins}{\mathrm{in}}
\newcommand{\out}{\mathrm{out}}
\newcommand{\ground}{\mathrm{zero}}
\newcommand{\mR}{\mathbb{R}}
\newcommand{\mL}{\mathcal{L}}

\title{A network dynamics approach to chemical reaction networks}
\author{A.J. van der Schaft\thanks{Arjan van der Schaft is with the Johann Bernoulli Institute for Mathematics and Computer
Science, University of Groningen, PO Box 407, 9700 AK, the Netherlands, +31-50-3633731, 
{\tt\small A.J.van.der.Schaft@rug.nl}}, S. Rao\thanks{Shodhan Rao is with Ghent University Global Campus,
119 Songdomunhwa-ro, Yeonsu-gu,
Incheon, South Korea 406-840, {\tt\small shodhan.rao@ghent.ac.kr}}, B. Jayawardhana
\thanks{Bayu Jayawardhana is with ENTEG, Faculty of Mathematics and
Natural Sciences, University of Groningen, the
Netherlands, +31-50-3637156, {\tt\small b.jayawardhana@rug.nl, bayujw@ieee.org}}
}

\date{}
\begin{document}
\maketitle

\begin{abstract} 
A crisp survey is given of chemical reaction networks from the perspective of general nonlinear network dynamics, in particular of consensus dynamics. It is shown how by starting from the complex-balanced assumption the reaction dynamics governed by mass action kinetics can be rewritten into a form which allows for a very simple derivation of a number of key results in chemical reaction network theory, and which directly relates to the thermodynamics of the system. Central in this formulation is the definition of a balanced Laplacian matrix on the graph of chemical complexes together with a resulting fundamental inequality. This directly leads to the characterization of the set of equilibria and their stability. Both the form of the dynamics and the deduced dynamical behavior are very similar to consensus dynamics, and provide additional insights and perspectives to the latter. 
The assumption of complex-balancedness is revisited from the point of view of Kirchhoff's Matrix Tree theorem.
Finally, using the classical idea of extending the graph of chemical complexes by an extra 'zero' complex, a complete steady-state stability analysis of mass action kinetics reaction networks with constant inflows and mass action outflows is given. This provides a unified framework for structure-preserving model reduction, and for the control (see already \cite{sontag}) and 'reverse engineering' of (bio-)chemical reaction networks.
\end{abstract}
\section{Introduction} \label{sec:intro}
Network dynamics has been the subject of intensive research in recent years due to the ubiquity of large-scale networks in various application areas. While many advances have been made in the analysis of linear network dynamics, the study of nonlinear network dynamics still poses many challenges, especially in the presence of in- and outflows.

In this paper, we revisit the analysis of chemical reaction networks as a prime example of nonlinear network dynamics, playing an important role in systems biology, (bio-)chemical engineering, and the emerging field of synthetic biology.  
Apart from being {\it large-scale} (typical reaction networks in living cells involve several hundreds of chemical species and reactions) a characteristic feature of chemical reaction network dynamics is their intrinsic {\it nonlinearity}. In fact, mass action kinetics, the most basic way to model reaction rates, leads to polynomial differential equations. On top of this, chemical reaction networks, in particular in a bio-chemical context, usually have inflows and outflows. 

The foundations of the structural theory of (isothermal) chemical reaction networks (CRNs) were laid in a series of seminal papers by Horn, Jackson, and Feinberg in the 1970s.
The basic starting point of  e.g. \cite{HornJackson, Horn, Feinberg2} is the identification of a graph structure for CRNs by defining the chemical complexes, i.e., the combination of chemical species appearing on the left-hand (substrate) and right-hand (product) sides of every reaction, as the vertices of a graph and the reactions as its edges. This enables the formulation of the dynamics of the reaction network as a dynamical system on the graph of complexes. Furthermore, in these papers the philosophy was put forward of delineating, by means of {\it structural} conditions on the graph, a large class of reaction networks exhibiting the same type of dynamics, irrespective of the precise values of the (often unknown or uncertain) reaction constants. This 'normal' dynamics is characterized by the property that for every initial condition of the concentrations there exists a unique positive equilibrium to which the system will converge. Other dynamics, such as multi-stability or presence of oscillations, can therefore only occur within reaction networks that are violating these conditions. 
The main sufficient structural conditions are known as the Deficiency Zero and Deficiency One theorems, see e.g. \cite{FeinbergHorn1974, Feinberg1}.
For an overview of results on CRNs, and current research in this direction including the global persistence conjecture, we refer to \cite{Angeli2009A} and the references quoted therein. An important step in extending the framework of CRNs towards {\it feedback stabilization} has been made in \cite{sontag}; also setting the stage for further {\it regulation} questions. 

\smallskip
The contribution of the present paper is two-fold. First, the formulation and analysis of mass action kinetics chemical reaction networks is revisited from the point of view of {\it consensus dynamics} and its nonlinear versions \cite{Cortes2008, Cortes2013, phsystemsgraphs}. The consideration of concepts from algebraic graph theory, such as the systematic use of weighted Laplacian matrices, provides a framework for (re-)proving many of the previously obtained results on CRNs in a much simpler and insightful manner. In particular, in our previous work \cite{vds, rao} we have shown how under the assumption of existence of a detailed-balanced equilibrium, or the weaker assumption of existence of a complex-balanced equilibrium (a concept dating back to Horn \& Jackson \cite{HornJackson}), the weights of the graph of complexes can be redefined in such a way that the resulting Laplacian matrix becomes symmetric (detailed-balanced case) or balanced (complex-balanced case). As a result, the characterization of the set of positive equilibria and their stability as originating in \cite{HornJackson, Horn, Feinberg2} follows in a simple way. Moreover, this formulation allows for a direct port-Hamiltonian interpretation \cite{Schaft2013Lyon}, merging CRNs with the geometric thermodynamical theory of Oster \& Perelson \cite{opk, op}, and leads to new developments such as a theory of structure-preserving model reduction of chemical reaction networks, based on Kron reduction of the Laplacian matrix \cite{vds, rao, RAO2014}.

Interestingly, the nonlinearity of the chemical reaction dynamics is closely related to the fact that apart from the graph of complexes, another construction comes into play, namely the {\it representation} of the graph of complexes into the space of concentration vectors. Only if this map is the identity (corresponding to networks with complexes consisting of single species), the chemical reaction dynamics is linear, and, under the assumption of complex-balancedness, reduces to standard linear consensus dynamics.


As indicated above, our approach is based on the assumption of existence of a complex-balanced equilibrium, generalizing the classical notion of a detailed-balanced equilibrium. Based on \cite{vdsJMC} a necessary and sufficient condition is discussed for the existence of a complex-balanced equilibrium based on the Matrix Tree theorem (a theorem going back to the work of Kirchhoff on electrical circuits), which extends the classical Wegscheider conditions for existence of a detailed-balanced equilibrium. We also make a connection with the property of {\it mass conservation}. Furthermore, we discuss how these results can be 'dualized' to consensus dynamics, providing new insights.

The second main contribution of the present paper is the dynamical analysis of chemical reaction networks with {\it inflows} and {\it outflows}. The extension of the stability theory of {\it equilibria} for reaction networks without inflows and outflows (called {\it closed} reaction networks in the sequel) to that of {\it steady states} for reaction networks {\it with} inflows and outflows (called {\it open networks}) is far from easy, due to the intrinsic nonlinearity of the reaction dynamics. Recently, there has been a surge of interest in open CRNs; we mention \cite{Angeli2009, Angeli2011a, Chaves2005, Craciun2010, Flach2010, RaoCDC2013}. 
In the present paper we analyze open reaction networks by revisiting\footnote{Recently also in \cite{Craciun2010} the idea of adding a zero complex was followed up; however in the different context of preclusion of multi-stability in open CRNs.} the classical idea of extending the graph of complexes by a `zero' complex \cite{HornJackson, FeinbergHorn1974}.  We will show how in this way the results based on complex-balancedness for closed CRNs can be fully extended to CRNs with constant inflows and mass action kinetics outflows\footnote{This class of open reaction networks is motivated by several examples of biochemical reaction network models (see for example, the biochemical model in \cite{Karen} and the examples mentioned in \cite{FeinbergHorn1974}), and also derives from the assumption that some of the complexes or species involved in the reactions are kept at a constant concentration (see e.g. \cite{Karen, Karen2}).}. In particular, while in \cite{HornJackson, FeinbergHorn1974} the specific properties of the zero complex do not play any role in the analysis, the present paper spells out the equivalence of steady states with equilibria of the extended network, and shows how due to the presence of inflows and outflows the set of steady states may shrink to a unique steady state, while furthermore the presence of steady states on the boundary of the positive orthant can be precluded. Moreover, it allows to extend the model reduction techniques of \cite{vds, rao, RAO2014} to CRNs with constant inflows and mass action kinetics outflows. 
The obtained stability analysis of steady states of open CRNs with constant inflows and mass action kinetics outflows is one of the, up to now rare, cases of a rather complete steady state analysis of nonlinear network dynamics with external inputs.
From a control perspective the steady state analysis of open CRNs opens the possibility of applying the internal model principle (see e.g. \cite{DEPERSIS2012}) to achieve output regulation for such systems with constant reference signals using proportional-integral controllers, for example, in the control of CSTR or gene-regulatory networks as in \cite{UHLENDORF2012}.   

\smallskip
\noindent

The structure of the paper is as follows. 
In Section 2, based on algebraic graph-theoretical tools explored in \cite{vds, rao}, we give a crisp overview of the theory of closed reaction network dynamics satisfying the complex-balanced assumption, based on rewriting the dynamics in terms of a balanced Laplacian matrix directly linked to its port-Hamiltonian formulation.  
We provide a new perspective on the characterization of complex-balancedness by the use of Kirchhoff's Matrix Tree theorem, and establish a connection to mass conservation. Furthermore, we indicate how ideas from reaction network dynamics may be applied to the context of consensus dynamics.
Section 3 deals with the detailed modeling of reaction networks having constant inflows and mass action kinetics outflows by extending the graph with an extra zero complex. Section 4 shows how the assumption of complex-balancedness can be extended to this case, and how this allows to derive precise results on the structure and stability of steady states. Section 5 provides a brief introduction to structure-preserving model reduction of open chemical reaction networks, based on Kron reduction of the graph of complexes. Conclusions follow in Section 6, while the Appendix describes how the situation of detailed-balanced chemical reaction networks can be understood as a special case of the complex-balanced case.
%
%

\smallskip
\noindent\emph{\bf Notation}:  
The space of $n$-dimensional real vectors consisting of all strictly positive entries is denoted by $\mR_+^{n}$ and the space of $n$-dimensional real vectors consisting of all nonnegative entries by $\bar{\mR}_+^{n}$. 
%
The mapping
$\Ln : \mathbb{R}_+^n \to \mathbb{R}^n, \quad x \mapsto \Ln (x),$
is defined as the mapping whose $i$-th component is given as
$\left(\Ln (x)\right)_i := \ln(x_i).$
Similarly, $\Exp : \mathbb{R}^n \to \mathbb{R}_+^n$ is the mapping whose $i$-th component is given as
$\left(\Exp (x)\right)_i := \exp (x_i)$. Furthermore, for two vectors $x,y \in \mathbb{R}_+^n$ we let $\frac{x}{y}$ denote the vector in $\mathbb{R}_+^n$ with $i$-th component $\frac{x_i}{y_i}$.
Finally $\mathds{1}_n$ denotes the $n$-dimensional vector with all entries equal to $1$, and $0_{n}$ the $n$-dimensional vector with all entries equal to zero, while $I_{n}$ is the $n \times n$ identity matrix. 

\smallskip
\noindent\emph{\bf Some graph-theoretic notions} (see e.g. \cite{Bollobas}): A directed graph\footnote{Sometimes called a {\it multigraph} since we allow for multiple edges between vertices.} $\mathcal{G}$ with $c$ vertices and $r$ edges is characterized by a $c \times r$ {\it incidence matrix}, denoted by $D$. Each column of $D$ corresponds to an edge of the graph, and contains exactly one element $1$ at the position of the head vertex of this edge and exactly one $-1$ at the position of its tail vertex; all other elements are zero. Clearly, $\mathds{1}^TD=0$. The graph is {\it connected} if any vertex can be reached from any other vertex by following a sequence of edges; direction not taken into account. It holds that $\rank D= c- \ell$, where $\ell$ is the number of {\it connected components} of the graph. In particular, $\mathcal{G}$ is connected if and only if $\ker D^T = \spa \mathds{1}$. The graph is {\it strongly connected} if any vertex can be reached from any other vertex, following a sequence of {\it directed} edges. A subgraph of $\mathcal{G}$ is a directed graph whose vertex and edge set are subsets of the vertex and edge set of $\mathcal{G}$. A graph is {\it acyclic} (does not contain cycles) if $\ker D =0$.
A {\it spanning tree} of a directed graph $\mathcal{G}$ is a connected, acyclic subgraph of $\mathcal{G}$ that spans all vertices of $\mathcal{G}$.

\section{Closed chemical reaction networks as dynamics on graphs}
%
\subsection{The complex graph formulation}
Consider a chemical reaction network with $m$ chemical species (metabolites) with concentrations $x \in \mathbb{R}^m_+$, among which $r$ chemical reactions take place. 
%
The graph-theoretic formulation, starting with the work of Horn, Jackson and Feinberg in the 1970s, is to associate to each complex (substrate as well as product) of the reaction network a vertex of a graph, while each reaction from substrate to product complex corresponds to a directed edge (with tail vertex the substrate and head vertex the product complex). 

Let $c$ be the total number of complexes involved in the reaction network, then the resulting directed graph $\mathcal{G}$ with $c$ vertices and $r$ edges is called the {\it graph of complexes}\footnote{In the literature sometimes also referred to as {\it reaction graphs}.}, and is defined by its $c \times r$ {\it incidence matrix} $D$.
Since each of the $c$ complexes is a combination of the $m$ chemical species we define the $m \times c$ matrix $Z$ with non-negative integer elements expressing the composition of the complexes in terms of the chemical species. The $k$-th column of $Z$ denotes the composition of the $k$-th complex, and the matrix $Z$ is called the {\it complex composition matrix}\footnote{In \cite{vds,rao} the matrix $Z$ was called the 'complex stoichiometric matrix'.}. It can be immediately verified that $ZD$ equals the standard stoichiometric matrix $S$. 
The mapping $Z: \mathbb{R}^c \to \mathbb{R}^m$ defines a {\it representation} \cite{godsil} of the graph of complexes $\mathcal{G}$ into the space $\mathbb{R}^m$ of chemical species (the $\alpha$-th vertex is mapped to the $\alpha$-th column of $Z$ in $\mathbb{R}^m$). Compared with other network dynamics the presence of the matrix $Z$ constitutes a major, and non-trivial, difference; especially in case $Z$ is not injective. The special case of $Z$ being the identity matrix corresponds to {\it single-species} substrate and product reaction networks ({\it SS reaction networks}). 

The dynamics of the reaction network  takes the form
\begin{equation}
\dot{x} = Sv(x)=ZDv(x) 
\end{equation}
where $v(x)$ is the vector of reaction rates. The most basic way to define $v(x)$ is {\it mass action kinetics}. For example, for the reaction $X_1 + 2X_2 \to X_3$ the mass action kinetics reaction rate is given as $v(x) = kx_1x_2^2$ with $k >0$ a reaction constant. In general, for a single reaction with substrate complex $\mathcal{S}$ specified by its corresponding column $Z_{\mathcal{S}} = \begin{bmatrix} 
Z_{\mathcal{S}1} & \cdots Z_{\mathcal{S}m} \end{bmatrix}^T$ of the complex composition matrix $Z$, the mass action kinetics reaction rate is given by
\[
kx_1^{Z_{\mathcal{S}1}}x_2^{Z_{\mathcal{S}2}} \cdots x_m^{Z_{\mathcal{S}m}},
\]
which can be rewritten as $k \exp (Z_{\mathcal{S}}^T \Ln x)$. 
Hence the reaction rates of the total reaction network are given by 
\[
v_j(x) = k_j \exp (Z_{\mathcal{S}_{j}}^T \Ln x), \quad j=1, \cdots,r,
\]
where $\mathcal{S}_{j}$ is the substrate complex of the $j$-th reaction with reaction constant $k_j >0$. 
This yields the following compact description of the total mass action kinetics rate vector $v(x)$. Define the $r \times c$ matrix $K$ as the matrix whose $(j,\sigma)$-th element equals $k_j,$
if the $\sigma$-th complex is the substrate complex for the $j$-th reaction, and zero otherwise. We will call $K$ the {\it outgoing co-incidence matrix} (since the $\sigma$-th column of $K$ specifies the weighted outgoing edges from vertex $\sigma$). Then 
\begin{equation}
v(x) = K \Exp (Z^T \Ln x),
\end{equation}
and the dynamics of the mass action kinetics reaction takes the form
\begin{equation}\label{closed}
\dot{x} = ZDK\Exp (Z^T \Ln x)
\end{equation}
The same expression (in less explicit form) was already obtained in \cite{sontag}.

It can be verified that the $c \times c$ matrix $L:= - DK$ has nonnegative diagonal elements and nonpositive off-diagonal elements. Moreover, since $\mathds{1}_m^TD=0$ also $\mathds{1}_m^TL=0$, i.e., the column sums of $L$ are all zero. Hence $L$ defines (a transposed version of) a weighted {\it Laplacian matrix}\footnote{In \cite{chapman} such a matrix $L$ was called an {\it out-degree} Laplacian matrix.}. From now on we will simply call $L= -DK$ the Laplacian matrix of the graph of complexes $\mathcal{G}$.
%

\subsection{Analysis of complex-balanced reaction network dynamics}
A chemical reaction network (\ref{closed}) is called {\it complex-balanced}  \cite{Horn} if there exists an equilibrium $x^* \in \mathbb{R}_+^m$, called a {\it complex-balanced equilibrium}, satisfying\footnote{In the special case $\im D \cap \ker Z = \{0\}$ ({\it deficiency zero} in the terminology of \cite{Feinberg2}) complex-balancedness is equivalent to the existence of a positive equilibrium of (\ref{closed}).}
\begin{equation}\label{bal}
Dv(x^*) = - L\Exp (Z^T \Ln (x^*)) =0
\end{equation}
Chemically (\ref{bal}) means that at the complex-balanced equilibrium $x^*$ not only the chemical species but also the complexes remain constant; i.e., for each complex the total inflow (from the other complexes) equals the total outflow (to the other complexes). 
Defining now the diagonal matrix
\begin{equation}\label{K}
\Xi(x^*) :=  \diag \big(\exp (Z_i^T \Ln (x^*))\big)_{i=1, \cdots, c},
\end{equation}
the dynamics (\ref{closed}) can be rewritten into the form
\begin{equation}\label{masterequation}
\dot{x} = - Z\mathcal{L}(x^*) \Exp (Z^T \Ln (\frac{x}{x^*})),\quad 
\mathcal{L}(x^*) := L \Xi(x^*),
\end{equation}
where, since $\Exp (Z^T \Ln (\frac{x^*}{x^*})) =  \mathds{1}_c$, the transformed\footnote{As shown in \cite{rao} the matrix $\mathcal{L}(x^*)$ is in fact  {\it independent} of the choice of the complex-balanced equilibrium $x^*$ up to a multiplicative factor for every connected component of $\mathcal{G}$.} Laplacian matrix $\mathcal{L}(x^*)$ satisfies
\begin{equation}
\mathcal{L}(x^*) \mathds{1}_c =0, \quad \mathds{1}_c^T\mathcal{L}(x^*) =0
\end{equation}
Hence $\mathcal{L}(x^*)$ is a {\it balanced} Laplacian matrix (column {\it and} row sums are zero). 
\begin{remark}
Under the stronger {\it detailed-balanced} assumption \cite{opk, op, vds} the Laplacian matrix $\mathcal{L}(x^*)$ is not only balanced, but in fact is {\it symmetric}. In the Appendix it is discussed how the detailed-balanced situation can be understood as a special case of the complex-balanced one.
\end{remark}
\begin{remark}\label{rem:WR}
Note that the vector $\Exp(Z^T \Ln(x^*))$ corresponding to a complex-balanced equilibrium $x^* \in \mathbb{R}^m_+$ defines a vector in $\mathbb{R}^c_+$ that is in the kernel of the Laplacian matrix $L$. It thus follows \cite[Lemma 3.2.9]{Gaterman} that the connected components of the graph $\mathcal{G}$ of a complex-balanced reaction network are {\it strongly connected}\footnote{Strong connectedness of the connected components is in CRN literature often referred to as {\it weak reversibility} \cite{Horn}}. 
This follows as well from the fact that a graph with balanced Laplacian matrix is strongly connected if and only if it is connected \cite{godsil}. The property also follows from Kirchhoff's Matrix Tree theorem to be discussed later on.
\end{remark}

%

As shown in \cite{rao} (generalizing the detailed-balanced scenario of \cite{vds}) a number of key properties of the reaction network dynamics can be derived in an insightful and easy way from the following fundamental fact. It is well-known \cite{Cortes2008} that balancedness of $\mathcal{L}(x^*)$ is equivalent to $\mathcal{L}(x^*) + \mathcal{L}^T(x^*)$ being positive semi-definite, i.e., $\alpha^T \mathcal{L}\alpha \geq 0$ for all $\alpha \in \mathbb{R}^c$. Based on {\it convexity} of the exponential function we can establish the following stronger property \cite{rao}. 
\begin{proposition}\label{fundamental}
$\gamma^T \mathcal{L}(x^*) \Exp (\gamma) \geq 0$ for any $\gamma \in \mathbb{R}^r$, with equality if and only if $D^T \gamma =0$.
\end{proposition}
This result leads to a direct proof of a number of key properties of (\ref{masterequation}), which are known within CRN theory but proven by tedious derivations. The first property which directly follows from Proposition \ref{fundamental}) is that all positive equilibria are in fact complex-balanced equilibria, and that given one complex-balanced equilibrium $x^*$ the set of {\it all} positive equilibria is given by
\begin{equation}\label{equilibria}
\mathcal{E} := \{ x^{**} \in \mathbb{R}^m_+ \mid S^T \Ln \left(x^{**} \right) = S^T \Ln \left(x^{*}\right) \}
\end{equation}
In particular, the set of positive equilibria $\mathcal{E}$ is a smooth manifold of dimension $m - \rank S$. Another property of $\mathcal{E}$ can be seen to be implied by the extra assumption of {\it mass conservation}, which is defined as follows.
\begin{definition}\label{defmassconservation}
A reaction network with complex composition matrix $Z$ and incidence matrix $D$ is said to satisfy {\it mass conservation} if there exists $\mu \in \mathbb{R}^m_+$ such that 
\begin{equation}\label{massconservation}
Z^T\mu \in \ker D^T,
\end{equation}
or, equivalently, $S^T \mu = D^TZ^T \mu =0$.
\end{definition}
\begin{remark}
In case $\mathcal{G}$ be connected $\ker D^T = \spa \mathds{1}$, and the definition of mass conservation reduces to the existence of $\mu \in \mathbb{R}^m_+$ such that $Z^T\mu=\mathds{1}$.
The vector $\mu$ specifies a vector of mass assignments ($\mu_i$ specifies the mass associated to the $i$-the chemical species), and the condition $Z^T\mu = \mathds{1}$ means that all complexes have identical mass. For the general case this holds on any connected component of $\mathcal{G}$.
\end{remark}
\begin{proposition}
Consider a chemical reaction network as before. Then the origin $0$ is on the boundary of $\mathcal{E}$ if and only if the chemical reaction network satisfies mass conservation. 
\end{proposition}
\begin{proof}
There exists a $x^{**} \in \mathbb{R}^m_+$ with all entries arbitrarily close to $0$ with $S^T\Ln(x^{**}) = S^T\Ln(x^{*})$ if and only if there exists a vector $z$ with all entries arbitrarily close to $- \infty$ such that $S^Tz = S^T\Ln(x^{*})$. This, in turn, holds if and only if there exists a positive vector $\mu \in \ker S^T = \ker D^TZ^T$, or, equivalently, $Z^T\mu \in \ker D^T$. 
\end{proof}

It directly follows from the structure of the Laplacian matrix $L$, see e.g. \cite{sontag, rao}, that the dynamics (\ref{closed}) leaves the positive orthant $\mathbb{R}_+^m$ invariant. Hence concentrations of chemical species remain positive for all future times. On the other hand, the possibility that the solution trajectories of (\ref{closed}) will approach the boundary of the positive orthant for $t \to \infty$ is not easily excluded. The reaction network is called {\it persistent}\footnote{It is generally believed that most reaction networks are persistent. However, up to now this {\it persistence conjecture} has been only proved in special cases (cf. \cite{Anderson}, \cite{Siegel}, \cite{Angeli2011} and the references quoted in there).} if for every $x_0 \in \mathbb{R}_+^{m}$ the $\omega$-limit set $\omega(x_0)$ of the dynamics (\ref{1}) does not intersect the boundary of $\bar{\mathbb{R}}_+^{m}$. 

Using a result from \cite{Feinberg1}, there exists for any initial condition $x_0 \in \mR^m_+$ a unique $x^{* *}\in \mathcal{E}$ such that $x^{**} - x_0 \in \im S$.  By using Proposition \ref{fundamental} in conjunction with the Lyapunov function
\begin{equation}\label{Gibbs}
G(x) =x^T \mathrm{Ln}\left(\frac{x}{x^{**}}\right) + \left(x^{**} - x \right)^T \mathds{1}_m
\end{equation}
it follows that the vector of concentrations $x(t)$ starting from $x_0$ will converge to $x^{**}$ if the reaction network is persistent. 
The chemical interpretation is that $G$ is (up to a constant) the {\it Gibbs' free energy} \cite{op, opk, vds}), with gradient vector $\frac{\partial G}{\partial x}(x) = \mathrm{Ln}\left(\frac{x}{x^{**}}\right)$ being the vector of {\it chemical potentials}. Hence (\ref{masterequation}) can be rewritten as
\begin{equation}\label{masterequation1}
\dot{x} = - Z\mathcal{L}(x^*) \Exp (Z^T \frac{\partial G}{\partial x}(x))
\end{equation}
and the`driving forces' of the reactions are seen to be determined by the {\it complex thermodynamical affinities}\footnote{See e.g. \cite{opk,vds,Schaft2013Lyon} for further information.} $\gamma (x) := Z^T \frac{\partial G}{\partial x}(x) = Z^T \Ln (\frac{x}{x^{**}})$. Furthermore, by Proposition \ref{fundamental} equilibrium arises whenever the components of $\gamma(x)$ reach `consensus' on every connected component of the graph of complexes $\mathcal{G}$.

\subsection{Port-Hamiltonian formulation}
The formulation (\ref{masterequation1}) admits a direct port-Hamiltonian interpretation (see e.g. \cite{vanderschaftmaschkearchive}, \cite{vanderschaftbook, NOW} for an introduction to port-Hamiltonian systems). Indeed, consider the auxiliary port-Hamiltonian system
\begin{equation}\label{PH}
\begin{array}{rcl}
\dot{x} & = & Zf \\[2mm]
e &= & Z^T \frac{\partial G}{\partial x}(x)
\end{array}
\end{equation}
with inputs $f \in \mR^c$ and outputs $e \in \mR^c$, and Hamiltonian given by the Gibbs' free energy $G$ defined in (\ref{Gibbs}). It follows from Proposition \ref{fundamental} that
\begin{equation}\label{resistive}
f = - \mathcal{L}(x^*) \Exp(e)
\end{equation}
defines a true {\it energy-dissipating} relation, that is, $e^T f \leq 0$ for all $e \in \mR^c$ and $f \in \mR^c$ satisfying (\ref{resistive}). By substituting (\ref{resistive}) into (\ref{PH}) one recovers the chemical reaction dynamics (\ref{masterequation1}). 

It should be noted that the energy-dissipating relation (\ref{resistive}) is intrinsically {\it nonlinear}, and generally can{\it not} be integrated to a relation of the form $f = - \frac{\partial R}{\partial e}(e)$ for some (Rayleigh) function $R: \mR^c \to \mR$, since the Poincar\'e integrability conditions are not satisfied (unless $Z$ is e.g. the identity matrix; see the SS reaction networks discussed later on). 


\subsection{Characterization of complex-balancedness}
Complex-balancedness can be characterized as follows, cf. \cite{vdsJMC} for further details. By the definition of $\Ln : \mathbb{R}_+^m \to \mathbb{R}^m$ the  existence of a complex-balanced equilibrium $x^* \in \mathbb{R}^m_+$, that is, $L \Exp (Z^T\Ln(x^*)) =0$, is equivalent to the existence of a vector $\mu^* \in \mathbb{R}^m$ such that
\begin{equation}
L \Exp (Z^T\mu^*)=0,
\end{equation}
or equivalently, $\Exp(Z^T\mu^*) \in \ker L$. Furthermore, we note that $\Exp(Z^T\mu^*) \in \mathbb{R}^c_+$.

First assume that the graph $\mathcal{G}$ is connected. Then the kernel of $L$ is $1$-dimensional, and a vector $\rho \in \mathbb{R}^c_+$ with $\rho \in \ker L$ can be computed by {\it Kirchhoff's Matrix Tree theorem}\footnote{This theorem goes back to the classical work of Kirchhoff on resistive electrical circuits \cite{Kirchhoff}; see \cite{Bollobas} for a succinct treatment. Nice accounts of the Matrix Tree theorem in the context of chemical reaction networks can be found in \cite{mirzaev, gunawardena}.}, which can be summarized as follows. 
Denote the $(i,j)$-th cofactor of $L$ by $C_{ij}=(-1)^{i+j}M_{i,j}$, where $M_{i,j}$ is the determinant of the $(i,j)$-th minor of $L$, which is the matrix obtained from $L$ by deleting its $i$-th row and $j$-th column. Define the adjoint matrix $\mathrm{adj}(L)$ as the matrix with $(i,j)$-th element given by $C_{ji}$. It is well-known that
$L \cdot \mathrm{adj}(L) = (\det L)I_c$, and since $\det L = 0$ this implies $L \cdot \mathrm{adj}(L)=0$.
Since $\mathds{1}^TL=0$ the sum of the rows of $L$ is zero, and hence by the properties of the determinant it is easily seen that $C_{ij}$ does not depend on $i$; implying that $C_{ij} = \rho_j, \, j=1, \cdots, c$. Hence the rows of $\mathrm{adj}(L)$ are given as the row vectors $\rho_j \mathds{1}^T, \, j=1, \cdots, c$, and by defining $\rho := (\rho_1, \cdots, \rho_c)$, it follows that $L\rho=0$. Furthermore, Kirchhoff's Matrix Tree theorem says (cf. \cite{Bollobas}, Theorem 14 on p.58) that $C_{ij} = \rho_i$ is equal to the sum of the products of weights of all the spanning trees of $\mathcal{G}$ directed towards vertex $i$. In particular, it follows that $\rho_k \geq 0, k=1, \cdots,c$. Moreover, since for every vertex $i$ there exists at least one spanning tree directed towards $i$ if and only if the graph is {\it strongly connected}, $\rho \in \mathbb{R}^c_+$ if and only if the graph is strongly connected. 
\begin{example}\label{excyclic}
Consider the cyclic reaction network\\
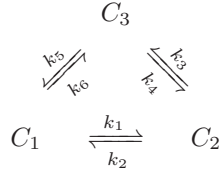
\begin{center}
\begin{tabular}{c c c}
& $C_3$ & \\
& {\rotatebox[origin=c]{45}{$\xleftrightharpoons[k_6]{\ k_5 \ }$}} \ {\rotatebox[origin=c]{-45}{$\xleftrightharpoons[k_4]{\ k_3 \ }$}} & \\
& $C_1$ \ \ \ \ $\xrightleftharpoons[k_2]{\ k_1 \ }$  \ \ \ \ $C_2$ &
\end{tabular}
\end{center}
\vspace{0.3cm}

\noindent in the three (unspecified) complexes $C_1, C_2, C_3$. The Laplacian matrix is given as
\[
L = \begin{bmatrix} k_1 + k_6 & - k_2 & - k_5 \\
- k_1 & k_2 + k_3 & -k_4 \\
-k_6 & - k_3 & k_4 + k_5
\end{bmatrix}
\]
By Kirchhoff's Matrix Tree theorem the corresponding vector $\rho$ satisfying $L \rho =0$ is given as
\[
\rho = \begin{bmatrix} k_3k_5 + k_2k_5 + k_2k_4 \\
k_1k_5 + k_1k_4 + k_4k_6 \\
k_1k_3 + k_3k_6 + k_2k_6 
\end{bmatrix},
\]
where each term corresponds to one of the three weighted spanning trees pointed towards the three vertices.
\end{example}
In case the graph $\mathcal{G}$ is not connected the same analysis can be performed on any of its connected components. 
\begin{remark}
The {\it existence} (not the explicit {\it construction}) of $\rho$ already follows from the Perron-Frobenius theorem \cite{Horn}, \cite[Lemma V.2]{sontag}; exploiting the fact that the off-diagonal elements of $-L:=DK$ are all nonnegative\footnote{This implies that there exists a real number $\alpha$ such that $-L + \alpha I_m$ is a matrix with all elements nonnegative. Since the set of eigenvectors of $-L$ and $-L + \alpha I_m$ are the same, and moreover by $\mathds{1}^TL=0$ there cannot exist a positive eigenvector of $-L$ corresponding to a non-zero eigenvalue, the application of Perron-Frobenius to $-L + \alpha I_m$ yields the result; see \cite[Lemma V.2]{sontag} for details.}.
\end{remark}
Returning to the existence of $\mu^* \in \mathbb{R}^m$ satisfying $L \Exp (Z^T\mu^*)=0$ this implies the following.
Let $\mathcal{G}_j, \, j=1, \cdots, \ell,$ be the connected components of the graph of complexes $\mathcal{G}$. For each connected component, define the vectors $\rho^1, \cdots, \rho^{\ell}$  as above by Kirchhoff's Matrix Tree theorem (i.e., as cofactors of $L$ or as sums of products of weights along spanning trees). 
Then define the total vector $\rho$ as the stacked column vector $\rho := \mathrm{col} ( \rho^1, \cdots, \rho^{\ell})$. 
Partition correspondingly the composition matrix $Z$ as $Z= [Z_1 \cdots Z_{\ell}]$.
Then there exists $\mu^* \in \mathbb{R}^m$ satisfying $L \Exp (Z^T\mu^*)=0$ if and only if each connected component is strongly connected and on each connected component
\begin{equation}
\Exp (Z_j^T\mu^*) = \beta_j \rho^j,\quad j=1, \cdots \ell,
\end{equation}
for some positive constants $\beta_j, j=1, \cdots \ell$. This in turn is equivalent to strong connectedness of each connected component of $\mathcal{G}$ and the existence of constants $\beta'_j$ such that
\begin{equation}
Z_j^T\mu^* = \Ln \rho^j + \beta'_j \mathds{1}, \quad j=1, \cdots, \ell
\end{equation}
Furthermore, this is equivalent to strong connectedness of each connected component, and
\begin{equation}\label{final}
\Ln \rho \in \im Z^T + \ker D^T
\end{equation}
Finally, (\ref{final}) is equivalent to
\begin{equation}\label{final1}
D^T\Ln \rho \in \im D^TZ^T = \im S^T
\end{equation}
Summarizing we have obtained
\begin{theorem}\label{Kirchhoff}
The reaction network dynamics $\dot{x} = - ZL\Exp (Z^T \Ln (x))$ on the graph of complexes $\mathcal{G}$ is complex-balanced if and only if each connected component of $\mathcal{G}$ is strongly connected (or, equivalently, $\rho \in \mathbb{R}^c_+$) and (\ref{final1}) is satisfied, where the coefficients of the sub-vectors $\rho^j$ of $\rho$ are obtained by Kirchhoff's Matrix Tree theorem for each $j$-th connected component of $\mathcal{G}$. 
\end{theorem}
\begin{remark} 
The easiest way to compute the elements $\rho_k, k=1, \cdots,c,$ of $\rho$ is by taking the determinant of the matrix obtained from $L$ by deleting its $k$-th row and $k$-th column.
\end{remark}
Clearly, if $\im Z^T = \mathbb{R}^c$, or equivalently $\ker Z=0$, then (\ref{final}) is satisfied for any\footnote{This is not surprising since $\ker Z=0$ implies zero-deficiency.} $\rho$.
\begin{corollary}
The reaction network dynamics $\dot{x} = - ZL\Exp (Z^T \Ln (x))$ is complex-balanced if and only if $\rho \in \mathbb{R}^c_+$ and 
\begin{equation}
\rho_1^{\sigma_1} \cdot \rho_2^{\sigma_2} \cdots \cdot \rho_c^{\sigma_c} =1,
\end{equation}
for all vectors $\sigma =\mathrm{col} (\sigma_1, \sigma_2, \cdots, \sigma_c) \in \ker Z \cap \im D $.
\end{corollary}
\begin{proof}
$\Ln \rho \in \im Z^T + \ker D^T$ if and only if $\sigma^T \Ln \rho=0$ for all $\sigma \in (\im Z^T + \ker D^T)^{\perp} = \ker Z \cap \im D$, or equivalently
\[
0=\sigma_1 \ln \rho_1 + \cdots + \sigma_c \ln \rho_c = \ln \rho_1^{\sigma_1} + \cdots +  \ln \rho_c^{\sigma_c} =
\ln (\rho_1^{\sigma_1} \cdots \rho_c^{\sigma_c})
\]
for all $\sigma \in \ker Z \cap \im D $.
\end{proof}
\begin{remark}
Note that the assumption of mass conservation (Definition \ref{defmassconservation}) may interfere with condition (\ref{final}). Indeed, mass conservation implies $\ker D^T \subset \im Z^T$ (unless $Z=0$), in which case (\ref{final}) reduces to $\Ln \rho \in \im Z^T$. 
\end{remark}

In the Appendix we will indicate how the constructive conditions for the existence of a complex-balanced equilibrium as obtained in Theorem \ref{Kirchhoff} relate to the classical Wegscheider conditions for the existence of a {\it detailed-balanced} equilibrium.

\subsection{SS reaction networks}
Reaction networks with single-species substrate and product complexes ({\it SS reaction networks}) correspond to $c=m$ and $Z=I_m$, in which case the dynamics (\ref{closed}) reduces to the {\it linear} dynamics
\begin{equation}\label{SS}
\dot{x} = DKx
\end{equation}
An SS reaction network is complex-balanced if there exists a positive equilibrium $x^* \in \mathbb{R}_+^m$ such that $DKx^{*}=0$, and hence can be rewritten as 
\begin{equation}\label{closedcbss}
\dot{x} = - \mathcal{L}(x^*) \frac{x}{x^*}, \quad  \mathcal{L}(x^*) := - DK\Xi(x^*), \quad
\Xi(x^*) := \diag (x_1^*, \cdots, x_m^*),
\end{equation}
where $ \mathcal{L}(x^*)$ is a balanced Laplacian matrix. 
%
The set of positive equilibria of a complex-balanced SS reaction network is given as $\mathcal{E} =\{ x^{**} \in \mathbb{R}^m_+ \mid D^T \Ln \left(x^{**} \right) = D^T \Ln \left(x^{*}\right) \}$, and thus, in case the graph $\mathcal{G}$ is connected, as
$ \mathcal{E} = \{ x^{**} \mid x^{**}= p x^*, p>0 \}$. 

In Remark \ref{rem:WR} we already mentioned that the existence of a complex-balanced equilibrium implies that the connected components of the graph are strongly connected. For SS reaction networks also the converse holds, as follows from the above discussion, either based on Kirchhoff's Matrix Tree theorem or on Perron-Frobenius theorem\footnote{Still another way is to make use of the result of \cite{Horn} stating that a mass action chemical reaction network is complex balanced if it is strongly connected and has zero deficiency. Recall that the deficiency is defined as $\rank D-\rank ZD$. Since $Z=I$ any SS network has zero-deficiency.}. 

\subsection{Relation with consensus dynamics}
The dynamics $\dot{x} = -Lx = DKx$ with $\mathds{1}^TL=0$ as occurring in SS reaction networks can be regarded as `dual' to the standard consensus dynamics $\dot{x} = -L_cx,$ where the Laplacian matrix $L_c$ satisfies $L_c \mathds{1}=0$. In a different context this has been explored in \cite{chapman} where $\dot{x} = -Lx$ with $\mathds{1}^TL=0$ was called {\it advection dynamics}. As also noted in \cite{chapman} this duality originates from a duality in the interpretation of the edges of the underlying directed graph $\mathcal{G}$. For $\dot{x} = -Lx$ with $\mathds{1}^TL=0$ an edge from vertex $i$ to $j$ denotes `{\it material flow}' from vertex $i$ to vertex $j$, while for $\dot{x} = -L_cx$ with $L_c\mathds{1}=0$ an edge from vertex $i$ to $j$ denotes `{\it information}' about vertex $i$ available at vertex $j$. Thus in the first case the graph $\mathcal{G}$ denotes a flow network, while in the latter case $\mathcal{G}$ is a communication graph.

It follows that the results described so far for flow networks can be `transposed' to communication graphs and consensus dynamics. First of all, the Laplacian $L_c$ with $L_c\mathds{1}=0$ can be expressed as $L_c = -J^TD^T$ where $D$ is again the incidence matrix of $\mathcal{G}$ while $J$ (dually to the matrix $K$ as before) can be called the {\it incoming co-incidence matrix}: the $i$-th column of $J$ specifies the weighted edges incoming to vertex $i$. Furthermore, the idea of transforming the `out-degree' Laplacian matrix $L = -DK$ to a balanced Laplacian matrix $\mathcal{L}(x^*)$ under the assumption of complex-balancedness of the graph (or equivalently, under the assumption of strong connectedness of its connected components) can be also applied to the consensus dynamics $\dot{x} = - L_cx$ with $L_c \mathds{1}=0$. Indeed, assume that the connected components of the graph $\mathcal{G}$ are strongly connected. Then Kirchhoff's Matrix Tree theorem provides a positive vector $\sigma \in \mathbb{R}^m_+$ such that $\sigma^TL_c=0$. In fact, $\sigma_j$ is given as the sum of the products of the weights along directed spanning trees directed {\it from} vertex $j$. It follows that $\frac{d}{dt}\sum_{j=1}^m \sigma_jx_j =0$, implying the conserved quantity $\sum_{j=1}^m \sigma_jx_j$. 
Defining the diagonal matrix $\Sigma:= \diag (\sigma_1, \cdots, \sigma_m)$ the transformed Laplacian matrix $\mathcal{L}_c := \Sigma L_c$ is balanced, and hence $\mathcal{L}^T_c  + \mathcal{L}_c \geq 0$. Note that this immediately yields an easy stability proof of the set of equilibria $\mathcal{E} = \{ x \in \mathbb{R}^m_+ \mid x = d \mathds{1}, d > 0 \}$ for the consensus dynamics $\dot{x} = -L_cx$. Indeed, the positive  function $V(x):= x^T\Sigma x$ satisfies
\begin{equation}
\frac{d}{dt} V(x) = -x^T (L^T_c \Sigma  + \Sigma L_c)x = -x^T (\mathcal{L}^T_c + \mathcal{L}_c)x \leq 0,
\end{equation}
and thus serves as a Lyapunov function proving asymptotic stability of the set of consensus states $\mathcal{E}$.
Furthermore, for any initial condition $x_0$ the dynamics will converge to the consensus state $d^* \mathds{1}$, where $d^*$ is given as $d^* = \frac{1}{m}\sum_{j=1}^m \sigma_j x_{0j}$, with $\sigma_1, \cdots, \sigma_m$ determined as above by Kirchhoff's Matrix Tree theorem. 
\section{Reaction networks with constant inflows and mass action kinetics outflows}\label{sec:ON}
In many cases of interest, including bio-chemical networks, reaction networks have inflows and outflows of chemical species. A mass action kinetics chemical reaction network with {\it constant inflows} and {\it mass action kinetics outflows} is described by the following extension\footnote{Note that (\ref{1}) formalizes a situation of {\it direct} in- and outflow of some of the chemical complexes in the reaction network. Modeling of in- or outflows of single chemical species which do not already appear as complexes in the graph need to be incorporated in (\ref{1}) by the introduction of extra complexes. For other scenarios of open chemical reaction reaction networks such as continuous-stirred tank reactors with convective in- and outflows we refer to e.g. \cite{Hangos}.} of (\ref{closed})
\begin{equation}\label{1}
\dot{x} = ZDv(x) + ZD_{\ins}v_{\ins} + ZD_{\out}v_{\out}(x), \quad x \in \mathbb{R}_+^m
\end{equation}
Here the matrices $D_{\ins}$ and $D_{\out}$ specify the structure of the inflows and outflows. $D_{\ins}$ is a matrix whose columns consist of exactly one element equal to $+1$ (at the row corresponding to the complex which has inflow) while the other elements are zero. Similarly, $D_{\out}$ is a matrix whose columns consist of exactly one element equal to $-1$ (at the row corresponding to the complex which has outflow) while the rest are zero. 
As in the closed network case, $v(x)$ is the vector of (internal) mass action kinetics reaction rates given by $v(x) = K\Exp (Z^T \Ln (x))$. Furthermore, $v_{\ins} \in \mathbb{R}_+^k$ is a vector of constant positive inflows, while $v_{\out}(x)  \in \mathbb{R}_+^l$ is a vector of mass action kinetics outflows described by mass action kinetics as
\begin{equation}\label{link1}
D_{\out}v_{\out}(x) = -\Delta_{\out} \Exp (Z^T \Ln (x)),
\end{equation}
where $\Delta_{\out}$ is a diagonal matrix with non-negative elements given by the mass action kinetics rate constants.

A classical idea due to \cite{HornJackson} is that by the addition of an extra complex the reaction network (\ref{1}) can be represented as a {\it closed} reaction network on the extended graph\footnote{Similar ideas of adding vertices to the graph are used in network flow theory; see e.g. \cite{Bollobas}.}. 
In fact, an extra complex is added in such a way that the edges from the extra complex to the ordinary complexes model the {\it inflows} into the network, while the edges towards the extra complex model the {\it outflows} of the network. The complex composition matrix $Z_{\ground}$ corresponding to the extra complex is defined to be the $m$-dimensional {\it zero column vector}, and the extra complex is therefore called the {\it zero complex}. Hence the zero complex serves as a combined `source and sink' complex, which does not contribute to the overall mass. As a consequence, the extended network cannot satisfy mass conservation, cf. Definition \ref{massconservation}.

The resulting graph, consisting of the original graph of complexes together with the zero complex, is called the {\it extended graph of complexes} of the open reaction network (\ref{1}), and has complex composition matrix
\[
Z_e = \begin{bmatrix} Z & Z_{\ground}\end{bmatrix}=\begin{bmatrix} Z & 0 \end{bmatrix}
\]
The incidence matrix of the extended complex graph, denoted by $D_e$, is given as
\[
D_e = \begin{bmatrix} B \\ B_{\ground} \end{bmatrix},
\]
with $B_{\ground}$ a row vector corresponding to the zero complex, while in the notation of (\ref{1})
\begin{equation}\label{eq:B}
B = \begin{bmatrix} D & D_{\ins} & D_{\out} \end{bmatrix}
\end{equation}
Now define the $c$-dimensional column vector $L_{\ins}$ as 
\begin{equation}\label{link}
L_{\ins} = - D_{\ins}v_{\ins}
\end{equation}
Furthermore, let $L_{\out}$ be the $c$-dimensional column vector whose $i$-th element is equal to minus the $i$-th diagonal element of $\Delta_{\out}$.
Then extend the $c \times c$ Laplacian matrix $L$ of the graph of (ordinary) complexes to an $(c + 1) \times (c + 1)$ Laplacian matrix $L_e$ of the extended graph of complexes as
\begin{equation}\label{Lapext}
L_e := \begin{bmatrix} 
L + \Delta_{\out} & L_{\ins} \\
L_{\out} & \delta_{\ins}
\end{bmatrix},
\end{equation}
where $\delta_{\ins} \geq 0$ equals minus the sum of the elements  of $L_{\ins}$. By construction $L_e$ has non-negative diagonal elements, non-positive off-diagonal elements, while its columns sums are all zero.

It follows that the dynamics (\ref{1}) of the mass action reaction network with constant inflows and mass action kinetics outflows is equal to the mass action kinetics dynamics of the {\it extended graph of complexes} with extended stoichiometric matrix $S_e=Z_eD_e$.
Indeed, since $Z_{\ground} = 0$
\begin{equation}
\dot{x} = Z[Dv(x) + D_{\ins}v_{\ins} + D_{\out}v_{\out}(x)] = ZBv_e(x) = Z_eD_ev_e(x) =S_ev_e(x),
\end{equation}
where 
\begin{equation}\label{eq:v_e}
v_e(x) = \begin{bmatrix}  v(x) \\ v_{\ins} \\ v_{\out}(x) \end{bmatrix}
\end{equation}
with $v(x) = K\Exp (Z^T \Ln (x))$.
Furthermore, by using (\ref{link}), (\ref{link1}), (\ref{Lapext}),
\begin{equation}\label{2}
\begin{array}{rcl}
\dot{x}  & =  & Z[Dv(x) + D_{\ins}v_{\ins} + D_{\out}v_{\out}(x)] \\[2mm]
& = & - \begin{bmatrix} Z & 0 \end{bmatrix} L_e \begin{bmatrix} \Exp (Z^T \Ln (x)) \\ 1 \end{bmatrix} =  - Z_eL_e \Exp (Z_e^T \Ln (x))
\end{array}
\end{equation}


\begin{example}\label{eg:CB}
As a simple example consider a reaction network with $x \in \mathbb{R}^3_+,$ consisting of one reversible reaction with forward and reverse reaction constants $k_+, k_- >0$, where there is a constant inflow $k_{\ins}$ towards the first complex $X_1$ and mass action kinetics outflow $k_{\out}x_2x_3^2$ out of complex $X_2+2X_3$, that is 
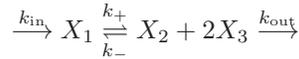
\[
\overset{k_{\ins}}\longrightarrow X_1 \overset{k_+}{\underset{k_{-}}\rightleftharpoons} X_2+2X_3 \overset{k_{\out}}\longrightarrow 
\]
The complex composition matrix $Z$ for this case is given by
\[
Z=\begin{bmatrix} 1 & 0 \\
0 & 1 \\
0 & 2
\end{bmatrix}
\]
The Laplacian matrix of the internal reversible reaction $($split into a forward and reverse reaction$)$ is
\[
L = \begin{bmatrix} k_+ & -k_-  \\ -k_+ & k_-   \end{bmatrix}.
\]
Together with the zero complex this corresponds to the Laplacian of the extended graph of complexes
\[
L_e = \begin{bmatrix} k_+ & -k_- & -k_{\ins} \\ -k_+ & k_-  + k_{\out} & 0 \\ 0 & - k_{\out} & k_{\ins} \end{bmatrix}
\]
and the following dynamics of the reaction network as in (\ref{1})
\[
\dot{x} =  Z\left(\begin{bmatrix} -k_+ & k_-  \\ k_+ & -k_-   \end{bmatrix} \begin{bmatrix} x_1 \\ x_2x_3^2 \end{bmatrix} + \begin{bmatrix} k_{\ins} \\0 \end{bmatrix} - \begin{bmatrix} 0 \\ k_{\out}x_2x_3^2 \end{bmatrix}\right ).
\]
\end{example}

\section{Analysis of reaction networks with constant inflows and mass action kinetics outflows}
%
Based on the formulation of the previous section we can extend the results concerning the stability of closed complex-balanced reaction networks as described before to the case of reaction networks with constant inflows and mass action kinetics outflows. 
As before we note that the representation (\ref{2}) implies that the positive orthant $\mathbb{R}_+^m$ is invariant for (\ref{1}). 
\begin{definition}\label{def:compbal}
An $x^* \in \bar{\mathbb{R}}^m_+$ is called a {\it steady-state} of the reaction network with constant inflows and mass action kinetics outflows given by $(\ref{1})$ if
\begin{equation}
Z[Dv(x^*) + D_{\ins}v_{\ins} + D_{\out}v_{\out}(x^*)] =0
\end{equation}
An $x^* \in \mR^m_+$ is called a {\it complex-balanced} steady-state if 
\begin{equation}\label{complexbalanced}
Dv(x^*) + D_{\ins}v_{\ins} + D_{\out}v_{\out}(x^*) =0
\end{equation}
If there exists a complex-balanced steady-state $x^* \in \mathbb{R}^m_+$ then the open reaction network $(\ref{1})$ is called {\it complex-balanced} .
\end{definition}
Note that, like in the case of closed networks, at a complex balanced steady state the total inflow from every complex is equal to the total outflow from it. 

The definition of a complex-balanced steady state $x^*$ can be succinctly written as $Bv_e(x^*) = 0$, with $B$ given by (\ref{eq:B}) and $v_e$ given by (\ref{eq:v_e}). We have the following simple but crucial observation showing that complex-balanced steady states for (\ref{1}) are actually complex-balanced equilibria of the extended network, and conversely.
\begin{proposition}\label{prop:equivalent}
$x^*$ is a complex-balanced steady-state, i.e., $Bv_e(x^*) = 0$, if and only if $D_ev_e(x^*) = 0$.
\end{proposition}
\begin{proof}
Since $\mathds{1}^TD_e=0$ the last row of $D_e$ is dependent on its first $c$ rows, that is, the rows of $B$. Hence $v_e(x^*) \in \ker D_e$ if and only if $v_e(x^*) \in \ker B$.
\end{proof}
\begin{remark} Note that this proposition does not remain true if we would consider instead of a single zero complex e.g. a source and a sink complex.
\end{remark}
If the network with constant inflows and mass action kinetics outflows has complex-balanced steady state $x^*$ then, similarly to (\ref{K}) for closed complex-balanced reaction networks, we define the diagonal matrix
\[
\Xi_e(x^*) :=  \diag \big(\exp (Z_i^T \Ln (x^*))\big)_{i=1, \cdots, c + 1}
\]
\begin{eqnarray*}
&=& \begin{bmatrix} \diag \big(\exp (Z_i^T \Ln (x^*))\big)_{i=1, \cdots, c} & 0 \\ 0 & 1 \end{bmatrix}
=: \begin{bmatrix} \Xi(x^*) & 0 \\ 0 & 1 \end{bmatrix}
\end{eqnarray*}
and rewrite
\begin{equation*}
D_ev_e(x) = - \mathcal{L}_e(x^*) \Exp \begin{bmatrix} Z^T \Ln (\frac{x}{x^*}) \\ 0 \end{bmatrix},
\end{equation*}
where 
\begin{equation}\label{eq:balLap}
\mathcal{L}_e(x^*) := L_e \Xi_e(x^*) = \begin{bmatrix} 
(L + \Delta_{\out}) \Xi(x^*) & L_{\ins} \\
L_{\out} \Xi(x^*) & \delta_{\ins}
\end{bmatrix}
\end{equation}
Note that $\Exp \begin{bmatrix} Z^T \Ln (\frac{x^*}{x^*})) \\ 0 \end{bmatrix} = \mathds{1}_{c+1}$. Hence, 
the existence of a complex-balanced steady state $x^*$ implies by Proposition \ref{prop:equivalent} that
$\mathcal{L}_e(x^*) \mathds{1}_{c+1} =0$. Hence, similarly to the previous section, $\mathcal{L}_e(x^*)$ satisfies $\mathds{1}_{c+1}^T \mathcal{L}_e(x^*) = 0$, as well as $\mathcal{L}_e(x^*) \mathds{1}_{c+1} =0$, and thus defines a {\it balanced weighted Laplacian matrix} for the extended graph of complexes.
The fact that the sum of the elements of the last row of $\mathcal{L}_e(x^*)$ is zero amounts to the equality
$L_{\out}\Xi(x^*) \mathds{1}_c + \delta_{\ins} = 0,$ or equivalently, 
\begin{equation}\label{balance2}
L_{\out} \Exp (Z^T \Ln(x^*)) = \mathds{1}_c^T L_{\ins},
\end{equation}
which can be interpreted as a {\it mass-balance} condition: at steady state the total inflow in the reaction network is equal to the total outflow.

\medskip
We obtain the following refined version of Proposition \ref{fundamental}.
\begin{theorem}\label{th}
Define $\mathcal{L}_e(x^*)$ as above. Then
\begin{equation}\label{dissipation}
\gamma_e^T \mathcal{L}_e(x^*) \Exp (\gamma_e) \geq 0
\end{equation}
for all $\gamma_e$, while equality holds if and only if $D_e^T \gamma_e =0$.
Furthermore, if $\gamma_e$ has last component zero, i.e., is of the form
\begin{equation}\label{gammae}
\gamma_e = \begin{bmatrix} \gamma \\ 0 \end{bmatrix},
\end{equation}
then equality holds if and only if $B^T \gamma =0$, or equivalently
\begin{equation}\label{equivalent}
D^T \gamma = 0, D^T_{\ins} \gamma = 0, D^T_{\out} \gamma = 0
\end{equation}
\end{theorem}
\begin{proof}
Only the last statement remains to be proved. This follows by noting that if $\gamma_e$ is given as in (\ref{gammae}) then $D_e^T \gamma_e =0$ if and only if $B^T \gamma =0$.
\end{proof}

We obtain the following basic theorem extending and refining the results of the previous section from closed networks to open reaction networks. 

\begin{theorem}\label{theorem}
Consider a mass action kinetics reaction network with constant inflows and mass action kinetics outflows $(\ref{1})$, for which there exists an $x^* \in \mR^m_+$ satisfying $(\ref{complexbalanced})$. Then

\medskip

$($1$)$: The set of positive steady states is given as
\begin{equation}\label{eq:char}
\{ x^{**} \in \mathbb{R}^m_+ \mid B^TZ^T \Ln (x^{**}) = B^TZ^T \Ln (x^{*}) \}.
\end{equation}
and all positive steady states are complex-balanced.

\medskip

$($2$)$: 
If every component of the graph of complexes is connected to the zero complex (or equivalently, if the extended graph of complexes is {\it connected}) then the set of steady states is given as
\[
\{ x^{**} \in \mathbb{R}^m_+  \mid Z^T \Ln (x^{**}) = Z^T \Ln (x^{*}) \}.
\]
In particular, if additionally $Z$ is surjective then the steady state $x^*$ is {\it unique}.

\medskip

$($3$)$: For every $x_0 \in \mathbb{R}_+^m$, there exists a unique $x_1 \in \mathcal{E}$ with $x_1-x_0 \in \im S$. The steady state $x_1$ is locally asymptotically stable with respect to initial conditions $x_0$ with $x_1 - x_0 \in \im S$. Furthermore, if the network is persistent then $x_1$ is globally asymptotically stable with respect to all these initial conditions.
\end{theorem}

\begin{proof}
({\it 1}): (\ref{eq:char}) follows from the characterization of the set of equilibria of a closed network, cf. (\ref{equilibria}), since the transpose of the stoichiometric matrix for the extended network is given as $S_e^T = D_e^TZ_e^T = B^TZ^T$. That every positive steady-state is complex balanced can be proved similar to the case of the closed networks case.

({\it 2}): Let $x^{**}$ be a positive steady state, and define $\gamma (x^{**}) =  Z^T \Ln \, (\frac{x^{**}}{x^*})$. By the first part of the theorem this means that $B^TZ^T \Ln (x^{**}) = B^TZ^T \Ln (x^{*})$, which is the same as $B^T\gamma (x^{**})  =0$, or equivalently (\ref{equivalent}). In particular, $D_{\ins}\gamma (x^{**})=0$ and $D_{\out}\gamma (x^{**})=0$, and thus the components of $\gamma (x^{**})$ corresponding to the complexes directly linked to the zero complex are zero. 
Furthermore, since $B^T\gamma (x^{**})  =0$, it follows that the components of $\gamma(x^{**})$ corresponding to each of the connected components of the extended graph are equal. Hence if the extended graph of complexes is connected, then $\gamma(x^{**})=0$, which is the same as
$Z^T \Ln (x^{**}) = Z^T \Ln (x^*)$. In particular, if $Z$ is surjective then this implies that the steady state $x^*$ is unique.

({\it 3}): This follows directly from the closed network case.
\end{proof}

It can be concluded from Theorem \ref{theorem} that the presence of inflows and outflows has the tendency to `shrink' the set of positive equilibria for the closed network to a smaller set of positive steady states; in fact, to a singleton if the extended graph is connected and $Z$ is surjective.
Furthermore, as shown in the following proposition, if the extended graph is connected, no steady states can occur at the boundary of the positive orthant $\mathbb{R}^m_+$, implying that the reaction network is automatically {\it persistent}.
\begin{proposition}
Consider a reaction network with constant inflows $v_{\ins} \in \mathbb{R}^k_+$ and mass action outflows $(\ref{1})$, which is complex-balanced. If the extended graph of complexes is connected, then there are no steady states at the boundary of $\mathbb{R}^m_+$.
\end{proposition}
\begin{proof}
Assume by contradiction that there exists a steady state $x_b \in \bar{\mR}_+^m$ with at least one component  (say the $i$-th one) equal to zero. Then consider a complex $\mathcal{C}$ containing this $i$-th species. Because $x_{bi} =0$ the outflows from complex $\mathcal{C}$ are zero, and by complex-balancedness this means that also all inflows to it are zero. From Remark \ref{rem:WR}, it follows that the extended graph of complexes is strongly connected. Hence there exists a {\it directed} path of reactions $\Pi$ starting from the zero complex and ending at $\mathcal{C}$. Now consider the complex which is preceding the complex $\mathcal{C}$ in this path. Then its outflows are zero, and therefore by complex-balancedness also its inflows.
Repeating this argument this shows that along $\Pi$ the inflow from the zero complex is zero, which yields a contradiction.
\end{proof}

\begin{example}
Consider the reaction network in Example \ref{eg:CB} with the Laplacian matrix of the extended graph of complexes given as
\[
L_e = \begin{bmatrix} k_+ & -k_- & -k_{\ins} \\ -k_+ & k_-  + k_{\out} & 0 \\ 0 & - k_{\out} & k_{\ins} \end{bmatrix}
\]
A complex-balanced steady-state $x^*= (x^*_1, x^*_2)$ satisfies the equations
\[
\begin{bmatrix} k_+ & -k_- \\ - k_+ & k_- + k_{\out} \end{bmatrix} \begin{bmatrix} x_1^* \\ x_2^*(x_3^*)^2 \end{bmatrix} = 
\begin{bmatrix} k_{\ins} \\ 0 \end{bmatrix}
\]
or more explicitly
\[
k_{\out}x_2^* (x_3^*)^2 = k_{\ins}, \quad k_+ x_1^* = k_{\ins} + k_-x_2^*(x_3^*)^2
\]
Hence the network is complex-balanced if $k_{\ins} \neq 0$ and $k_{\out} \neq 0$ (and also in the degenerate case $k_{\ins} = k_{\out} = 0$).
The mass-balance condition $(\ref{balance2})$ in this case amounts to
\begin{equation}\label{balance3}
k_{\ins} = k_{\out}x_2^* (x_3^*)^2.
\end{equation}

If $k_{\ins} \neq 0$ and $k_{\out} \neq 0$ then the set of steady states of the network is $1$-dimensional. Note on the other hand that the set of equilibria for the case {\it without} inflows and outflows ($k_{\ins} = k_{\out} = 0$) is $2$-dimensional; in line with the observation that the addition of inflows and outflows has the tendency to shrink the set of steady states as compared to the set of equilibria.
Finally, note that the matrix $\mathcal{K}_e(x^*)$ in this example equals $\diag (x_1^*, x_2^* (x_3^*)^2,1)$, while the resulting matrix $\mathcal{L}_e(x^*)$ is given by 
\[
\mathcal{L}_e(x^*) =
\begin{bmatrix} k_+x_1^* & -k_-x_2^* (x_3^*)^2 & -k_{\ins} \\ -k_+x_1^* & (k_-  + k_{\out})x_2^* (x_3^*)^2 & 0 \\ 0 & - k_{\out}x_2^*(x_3^*)^2 & k_{\ins} \end{bmatrix}
\]
%
%

\end{example}

\section{Structure-preserving model reduction of open reaction networks}
As detailed in the previous section, the dynamics of a complex-balanced chemical reaction network with constant inflows and outflows governed by mass action kinetics ('open reaction network') can be written in terms of the balanced Laplacian matrix given by (\ref{eq:balLap}) as follows
\begin{equation}\label{eq:open}
\dot{x}=-Z_e\mathcal{L}_e(x^*)\Exp(Z_e^T\Ln(\frac{x}{x^*})).
\end{equation}
This specific form allows for application of the model reduction method discussed in \cite{RAO2014,rao}; see also \cite{vds} for the detailed-balanced case. This method is inspired by the Kron reduction method of resistive electrical networks described in \cite{Kron}; see also \cite{vdsSCL}. The speciality of this method is that it is structure-preserving in the sense that the reduced model corresponds to a complex-balanced chemical reaction network governed by mass action kinetics just like the original model. To make the paper self-contained, we briefly describe the method below. For a detailed description of the method, the reader is referred to \cite{RAO2014,rao}.

Let $\mathcal{V}$ denote the set of vertices of the graph of complexes. We perform model reduction by {\it deleting certain complexes in the graph of complexes}, resulting in a reduced graph of complexes. We ensure that the set of complexes that are deleted does not include the zero complex, because otherwise the reduced network corresponding to an open network would become a closed network. Deletion of a complex is equivalent to imposing the complex-balancing condition on it, i.e., the condition that the net inflow into the complex is equal to the net outflow from it. Consider a subset $\mathcal{V}_o \subset \mathcal{V}$ of dimension $c+1-\hat{c}$ that we wish to delete in order to reduce the model. Without loss of generality, assume that the first $\hat{c}$ rows and columns of $\mL_e(x^*)$ and the first $\hat{c}$ columns of $Z_e$ correspond to $\mathcal{V}_r:=\mathcal{V}\backslash \mathcal{V}_o$. Consider the resulting partition of $\mL_e(x^*)$ given by
\begin{equation}\label{eq:part}
\mL_e(x^*)=\begin{bmatrix}
\mL_{11}(x^*) & \mL_{12}(x^*)\\
\mL_{21}(x^*) & \mL_{22}(x^*)
\end{bmatrix}
\end{equation}
where $\mL_{11}(x^*) \in \mR^{\hat{c}\times \hat{c}}$, $\mL_{12}(x^*) \in \mR^{\hat{c} \times (c+1-\hat{c})}$, $\mL_{21}(x^*)\in \mR^{(c+1-\hat{c})\times \hat{c}}$ and $\mL_{22}(x^*)\in \mR^{(c+1-\hat{c})\times(c+1-\hat{c})}$, and the resulting partition of $Z
_e$ given by $Z_e = \begin{bmatrix} Z_1 & Z_2 \end{bmatrix}$. Then write out the dynamics (\ref{eq:open}) as
\[
\dot{x} = - \begin{bmatrix} Z_1 & Z_2 \end{bmatrix}  \begin{bmatrix} \mL_{11}(x^*) & \mL_{12}(x^*) \\ \mL_{21}(x^*) & \mL_{22}(x^*) \end{bmatrix} 
\begin{bmatrix} \Exp \left( Z_1^T \Ln \left(\frac{x}{x^*}\right)\right) \\ \Exp \left(Z_2^T \Ln \left(\frac{x}{x^*}\right)\right) \end{bmatrix}
\]
Let $\hat{\mL}_e(x^*)$ denote the Schur complement of $\mL_e(x^*)$ with respect to the indices corresponding to $\mathcal{V}_o$. Consider now the auxiliary dynamical system
\[
\begin{bmatrix} \dot{y}_1 \\ \dot{y}_2 \end{bmatrix} = - \begin{bmatrix} \mL_{11}(x^*) & \mL_{12}(x^*) \\ \mL_{21}(x^*) & \mL_{22}(x^*) \end{bmatrix} 
\begin{bmatrix} w_1 \\ w_2 \end{bmatrix}
\]
Note that the complex-balancing condition on the complexes in $\mathcal{V}_o$ can be imposed by setting the constraint $\dot{y}_2 =0$. This results in the equation 
\[
w_2 = - \mL_{22}(x^*)^{-1}\mL_{21}(x^*)w_1, 
\]
leading to the reduced auxiliary dynamics defined by the Schur complement
\begin{equation}
\dot{y}_1 = - \big( \mL_{11}(x^*) - \mL_{12}(x^*)\mL_{22}(x^*)^{-1}\mL_{21}(x^*) \big) w_1 
= - \hat{\mL}_e(x^*) w_1
\end{equation}
Substituting $w_1 = \Exp \big(Z_1^T \Ln(x)\big)$ in the above equation and making use of 
$\dot{x} = Z_1 \dot{y}_1 + Z_2 \dot{y}_2 = Z_1 \dot{y}_1$, we then obtain the reduced model given by
\begin{equation}\label{reduced}
\quad \dot{x} = - \hat{Z}_e \hat{\mL_e}(x^*) \Exp \left(\hat{Z}_e^T \Ln \left(\frac{x}{x^*}\right)\right).
\end{equation}
where $\hat{Z}_e:=Z_1$. The following proposition ensures that $\hat{\mL_e}(x^*)$ obeys all the properties of the weighted Laplacian matrix of a complex-balanced reaction network corresponding to a graph of complexes with vertex set $\mathcal{V}_r$. 
\smallskip  

\begin{proposition}\label{prop:WL}
Consider an open complex-balanced network with dynamics given by equation (\ref{eq:open}). With $\mathcal{V}$,  $\mathcal{V}_o$ and $\hat{\mL}_e$ as defined above, the following properties hold:
\begin{enumerate}
\item All diagonal elements of $\hat{\mL}_e(x^*)$ are positive and off-diagonal elements are nonnegative.
\item $\mathds{1}_{\hat{c}}^T\hat{\mL}_e(x^*)=0$ and $\hat{\mL}_e(x^*)\mathds{1}_{\hat{c}}=0$, where $\hat{c}:=c+1-\text{dim}(\mathcal{V}_o)$.
\end{enumerate} 
If $\mathcal{E}$ and $\hat{\mathcal{E}}$ denote the set of steady-states of the original and the reduced networks described by (\ref{eq:open}) and (\ref{reduced}) respectively, then $\mathcal{E} \subseteq \hat{\mathcal{E}}$.
\end{proposition}
\begin{proof}
See proofs of \cite[Propositions 5.1,5.2]{rao}
\end{proof}

From Proposition \ref{prop:WL}, it follows that the reduced network (\ref{reduced}) corresponding to a complex balanced network (\ref{eq:open}) is also complex balanced. Complexes belonging to a certain connected component remain in the same connected component if not deleted. Thus mass conservation is preserved under our model reduction procedure.

Finally, we remark that for the application of our model reduction method, one can also start directly from the form of equations (\ref{2}) given by
$
\dot{x}=Z_eL_e\Exp(Z_e^T\Ln(x))
$, instead of the form (\ref{eq:open}) that uses the balanced Laplacian. Let $\hat{L}_e$ denote the Schur complement of $L_e$ with respect to the indices corresponding to $\mathcal{V}_o$. Consider the reduced model given by
\[
\dot{x}=\hat{Z}_e\hat{L}_e\Exp\big(\hat{Z}_e^T\Ln(x)\big)
\]
and note that it is the same as the reduced model (\ref{reduced}).
\section{Conclusions and outlook} 
We have discussed mass action kinetics chemical reaction networks as a challenging example of large-scale and nonlinear network dynamics, and have pointed out similarities with (nonlinear versions of) consensus dynamics. A fundamental difference resides in the complex composition matrix $Z$, which defines a representation of the graph of complexes (into the space of chemical species). 
Kirchhoff's Matrix Tree theorem has been discussed as an insightful way to compute the kernel of the Laplacian matrix, which, among others, yields an explicit characterization of the existence of a complex-balanced equilibrium. Also the relation to mass conservation has been pointed out.

For a particular class of open reaction networks, namely those with constant inflows and mass action outflows, a detailed stability analysis has been obtained by exploiting the notion of zero complex. By using the graph-theoretical techniques that we have used earlier to analyze closed complex-balanced reaction networks \cite{rao} this leads to a complete steady state stability analysis for this class of open reaction networks. Our results imply the intuitively obvious fact that the presence of inflows and outflows has the tendency to shrink the set of positive equilibria to a smaller set of positive steady states, and leads to the vanishing of possible steady states at the boundary of the positive orthant. This can be related to the feedback stabilization problem studied in \cite{sontag}, as well as internal model control.

An important extension of our results concerns the consideration of other types of kinetics, in particular Michaelis-Menten kinetics; see already \cite{melbourne} for the closed network case. Furthermore, the framework described in this paper can serve as a starting point for the inclusion of {\it regulatory networks} thus leading to direct control and 'reverse engineering' questions.

\section{Appendix: Detailed-balanced reaction networks}
The assumption of existence of a complex-balanced equilibrium can be strengthened to the existence of a {\it detailed-balanced} equilibrium. In this case we start with a directed graph of complexes $\mathcal{H}$ with $c$ complexes and $p$ edges, where each edge corresponds to a {\it reversible} reaction, cf. \cite{vds}. Assuming again mass action kinetics, the reaction rate $v^r_j(x)$ of each $j$-th reversible reaction is given as the difference
\begin{equation*}
v^r_j(x) = k^+_j \exp (Z_{\mathcal{S}_{j}}^T \Ln x) - k^-_j \exp (Z_{\mathcal{P}_{j}}^T \Ln x),
\end{equation*}
where $\mathcal{S}_{j}$ is the substrate and $\mathcal{P}_{j}$ the product complex, and where $k^+_j$ and $k^-_j$ are respectively the {\it forward} and {\it reverse} reaction constants of the reversible reaction. 
Note that $v^r_j(x)$ may take positive and negative values, in contrast with the previously considered case of irreversible reaction rates $v_j(x) \geq 0$. 
A reversible reaction network can be brought into the irreversible form as discussed before by defining the directed graph $\mathcal{G}$ as having the same vertex set as $\mathcal{H}$ but with twice as many edges: every edge $(i,j)$ of $\mathcal{H}$ is split into two edges (of opposite orientation) $(i,j)$ and $(j,i)$ of $\mathcal{G}$. 

A reversible mass action kinetics reaction network with graph of complexes $\mathcal{H}$ is called {\it detailed-balanced} if there exists an $x^* \in \mathbb{R}_+^m$ satisfying $v^r(x^*) = 0$, i.e.
\[
k_{j}^+ \exp \left(Z_{\mathcal{S}_{j}}^T \Ln (x^*)\right) = k_{j}^- \exp \left(Z_{\mathcal{P}_{j}}^T \Ln (x^*)\right), \, j=1, \cdots,r
\]
It is immediate that detailed-balancedness is a special case of complex-balancedness, with the reaction rates in the two opposite edges of $\mathcal{G}$ corresponding to a single edge of $\mathcal{H}$ being equal\footnote{Thermodynamically the assumption of detailed-balancedness is well-justified; it corresponds to microscopic reversibility \cite{opk}.}. (Instead of having the total sum of inflows to be equal to the total sum of outflows for every complex.)

Defining the {\it equilibrium constants} $K^{\mathrm{eq}}_j = \frac{k_{j}^+}{k_{j}^-}$ of each reversible reaction, and the vector $K^{\mathrm{eq}} := (K^{\mathrm{eq}}_1, \cdots, K^{\mathrm{eq}}_r)^T$, it can be shown \cite{FeinbergWeg,Schuster,vds} that detailed-balancedness is equivalent to the Wegscheider conditions
\[
\Ln K^{\mathrm{eq}} \in \im D^T_{\mathcal{H}}Z^T= \im S_{\mathcal{H}}^T,
\]
with $D_{\mathcal{H}}$ the incidence matrix of the graph $\mathcal{H}$, and $S_{\mathcal{H}}$ the stoichiometric matrix corresponding to $\mathcal{H}$ (i.e., every column of $S_{\mathcal{H}}$ corresponds to a {\it reversible reaction}). The assumption of detailed-balancedness implies that {\it all} equilibria are actually detailed-balanced, and that we may define the {\it conductances} of the $j$-th reversible reaction as
\begin{equation}
\kappa_{j}(x^*) := k_{j}^+ \exp \left(Z_{\mathcal{S}_{j}}^T \Ln (x^*)\right) = k_{j}^- \exp \left(Z_{\mathcal{P}_{j}}^T \Ln (x^*)\right), \, j=1, \cdots,r
\end{equation}
(See \cite{Ederer2007,Schaft2013Lyon} for a discussion regarding the similarities of these constants with conductances in other physical networks.)
It is readily seen that the detailed-balanced assumption is equivalent to the transformed Laplacian matrix $\mathcal{L}(x^*) := L \Xi(x^*)$ for the graph $\mathcal{G}$ being {\it symmetric}, with the $(i,j)$-th = $(j,i)$-th element being equal to the conductance of the reversible reaction between the $i$-th and the $j$-th complex. This means that $\mathcal{L}(x^*)$ can be written as
\[
\mathcal{L}(x^*)=D_{\mathcal{H}}\mathcal{K}^r(x^*)D_{\mathcal{H}}^T
\]
where\footnote{It can be shown \cite{vds} that the matrix $\mathcal{K}^r(x^*)$ is {\it independent} of the choice of the thermodynamic equilibrium $x^*$ up to multiplicative factor for every connected component of $\mathcal{H}$.} $\mathcal{K}^r(x^*)$ is the diagonal matrix of conductances $\kappa_{j}(x^*), j=1, \cdots,r$.
Hence \cite{vds} the dynamics takes the form
\begin{equation*}
\dot{x} = - Z D_{\mathcal{H}}\mathcal{K}^r(x^*)D_{\mathcal{H}}^T \Exp (Z^T \Ln (\frac{x}{x^*})).
\end{equation*}
For $Z=I$ this amounts to {\it symmetric} consensus dynamics on the graph $\mathcal{H}$ {\it without} its orientation.


Finally we will indicate the connections of the Wegscheider conditions mentioned above to the characterization of complex-balancedness as obtained by the Kirchhoff's Matrix Tree theorem. For further details we refer to \cite{vdsJMC}.
Consider a complex-balanced reaction network, with Laplacian matrix $L=-DK$. Compute based on Kirchhoff's Matrix Tree theorem the vector $\rho \in \mathbb{R}^c_+$ satisfying $L\rho=0$, leading to the transformed balanced Laplacian matrix $\mathcal{L}$ given as $\mathcal{L} = L \diag (\rho_1, \cdots, \rho_c)$. In case the reaction network is {\it detailed-balanced} it follows that this transformed Laplacian matrix $\mathcal{L}$ is actually {\it symmetric}; instead of just balanced. 
Symmetry of $\mathcal{L}$ can be seen to be equivalent to the {\it weakened} Wegscheider conditions\footnote{As shown in \cite{vdsJMC} the weakened Wegscheider conditions are also equivalent to the notion of {\it formal balancing} introduced in \cite{dickinson} as a formalization of the 'circuit conditions' of \cite{FeinbergWeg}.} (only depending on the structure of the graph of complexes, and {\it not} on the composition of the complexes) 
\[
\Ln K^{\mathrm{eq}} \in \im D_{\mathcal{H}}^T
\]
\begin{example}
Consider the network described in Ex. \ref{excyclic}. The transformed Laplacian matrix is computed as
\[
\begin{array}{rcl}
\mathcal{L} = \begin{bmatrix} k_1 + k_6 & - k_2 & - k_5 \\
- k_1 & k_2 + k_3 & -k_4 \\
-k_6 & - k_3 & k_4 + k_5
\end{bmatrix}
\begin{bmatrix} k_3k_5 + k_2k_5 + k_2k_4 & 0 & 0 \\
0 & k_1k_5 + k_1k_4 + k_4k_6 & 0 \\
0 & 0 & k_1k_3 + k_3k_6 + k_2k_6 
\end{bmatrix} = \\[6mm]
\begin{bmatrix}
(k_1+k_3)(k_3k_5 + k_2k_5 + k_2k_4) & -k_2(k_1k_5 + k_1k_4 + k_4k_6) & -k_5(k_1k_3 + k_3k_6 + k_2k_6) \\
-k_1(k_3k_5 + k_2k_5 + k_2k_4) & (k_2 + k_3)(k_1k_5 + k_1k_4 + k_4k_6) & -k_4(k_1k_3 + k_3k_6 + k_2k_6) \\
-k_6(k_3k_5 + k_2k_5 + k_2k_4) & -k_3(k_1k_5 + k_1k_4 + k_4k_6) & (k_4 + k_5)(k_1k_3 + k_3k_6 + k_2k_6) 
\end{bmatrix}
\end{array}
\]
This matrix is symmetric if and only if
\[
k_1k_3k_5 = k_2k_4k_6
\]
On the other hand, $\Ln K^{\mathrm{eq}} \in D_{\mathcal{H}}^T$ amounts to
\[
\begin{bmatrix} \ln \frac{k_1}{k_2} \\ \ln \frac{k_3}{k_4} \\ \ln \frac{k_5}{k_6} \end{bmatrix} \in
\im \begin{bmatrix}-1 & 0 & 1 \\ 1 & -1 & 0 \\ 0 & 1 & -1 \end{bmatrix}
\]
which reduces to $\ln \frac{k_1}{k_2} + \ln \frac{k_3}{k_4} + \ln \frac{k_5}{k_6} =0$, and hence to the same condition $k_1k_3k_5 = k_2k_4k_6$.
\end{example}

\end{document}